\documentclass [12pt,leqno,fleqn]{amsart}

\usepackage{amsmath,amstext,amsthm,amsxtra,mathtools,amssymb}

\usepackage[colorlinks=true,citecolor=blue,pagebackref,citecolor=magenta, urlcolor=cyan , hypertexnames=true]{hyperref}

\usepackage[normalem]{ulem}
\usepackage{txfonts} 
\usepackage[T1]{fontenc}
\usepackage{lmodern}

 \usepackage{euler}   
\usepackage{todonotes}

\usepackage[backrefs,msc-links,nobysame,initials]{amsrefs}

\allowdisplaybreaks

\theoremstyle{plain}

\newtheorem{lemma}[equation]{Lemma}

\newtheorem{theorem}[equation]{Theorem}
\newtheorem{corollary}[equation]{Corollary}
\theoremstyle{definition}
\newtheorem{definition}[equation]{Definition}

\numberwithin{equation}{section}

\def\norm#1.#2.{\lVert#1\rVert_{#2}}
\def\Norm#1.#2.{\bigl\lVert#1\bigr\rVert_{#2}}
\def\NOrm#1.#2.{\Bigl\lVert#1\Bigr\rVert_{#2}}
\def\NORm#1.#2.{\biggl\lVert#1\biggr\rVert_{#2}}
\def\NORM#1.#2.{\Biggl\lVert#1\Biggr\rVert_{#2}}

\def\ind{\textnormal{\textbf 1}}

\def\R{\mathbb R}

\def\Ms { \mathsf M _\text{S}}
\def\MM { \mathsf M _1}
\def\MMw {\mathsf M_1 ^\mu}
\def\Mw{\mathsf M  _\text{S} ^{w} }

\def\Cs{\mathsf C_{\mathsf {S},w}}
\def\Cw{\mathsf C _{\mathsf{S}} ^w}
\def\CC{\mathsf C _{1,w}}
\def\CCw{\mathsf C _1 ^\mu}

\DeclareMathOperator*{\esssup}{ess\,sup}

\def\ip#1,#2,{\langle #1,#2\rangle}
\def\Ip#1,#2,{\bigl\langle#1,#2\bigr\rangle}
\def\IP#1,#2,{\Bigl\langle#1,#2\Bigr\rangle}

\def\abs#1{\lvert#1\rvert}

\begin{document}
\raggedbottom
\title[Weighted Solyanik estimates]{Weighted Solyanik estimates for the strong maximal function}

\author{Paul Hagelstein}
\address{Department of Mathematics, Baylor University, Waco, Texas 76798}
\email{\href{mailto:paul_hagelstein@baylor.edu}{paul\!\hspace{.018in}\_\,hagelstein@baylor.edu}}
\thanks{P. H. is partially supported by a grant from the Simons Foundation (\#208831 to Paul Hagelstein).}

\author{Ioannis Parissis}
\address{Department of Mathematics, Aalto University, P. O. Box 11100, FI-00076 Aalto, Finland}
\email{\href{mailto:ioannis.parissis@gmail.com}{ioannis.parissis@gmail.com}}
\thanks{I. P. is supported by the Academy of Finland, project 277008.}


\begin{abstract} Let $\Ms$ denote the strong maximal operator on $\R^n$ and let $w$ be a non-negative, locally integrable function. For $\alpha\in(0,1)$ we define the  weighted sharp Tauberian constant $\Cs$ associated with $\Ms$ by
\[
 \Cs(\alpha)\coloneqq \sup_{\substack {E\subset \R^n \\ 0<w(E)<+\infty}}\frac{1}{w(E)}w(\{x\in\R^n:\, \Ms(\ind_E)(x)>\alpha\}).
\]
We show that $\lim_{\alpha\to 1^-} \Cs(\alpha)=1$  if and only if $w\in A_\infty ^*$, that is if and only if $w$ is a \emph{strong Muckenhoupt weight}. This is quantified by the estimate $\Cs(\alpha)-1\lesssim_{n} (1-\alpha)^{ (cn [w]_{A_\infty ^*})^{-1}}$  as $\alpha\to 1^-$, where $c>0$ is a numerical constant; this estimate is sharp in the sense that the exponent $1/(cn[w]_{A_\infty ^*})$ can not be improved in terms of $[w]_{A_\infty ^*}$.  As corollaries, we obtain a sharp reverse H\"older inequality for strong Muckenhoupt weights in $\R^n$ as well as a quantitative imbedding of $A_\infty^*$ into $A_{p}^*$.  We also consider the strong maximal operator on $\R^n$ associated with the weight $w$ and denoted by $\Mw$. In this case the corresponding sharp Tauberian constant $\Cw$ is defined by
\[
 \Cw(\alpha)  \coloneqq \sup_{\substack {E\subset \R^n \\ 0<w(E)<+\infty}}\frac{1}{w(E)}w(\{x\in\R^n:\, \Mw(\ind_E)(x)>\alpha\}).
\]
We show that there exists some constant $c_{w,n}>0$ depending only on $w$ and the dimension $n$ such that $\Cw(\alpha)-1 \lesssim_{w,n} (1-\alpha)^{ c_{w,n} }$ as $\alpha\to 1^-$  whenever $w\in A_\infty ^*$ is a strong Muckenhoupt weight.
\end{abstract}
\maketitle

\section{Introduction} \label{s.intro} We are interested in asymptotic estimates for the distribution functions of maximal functions and allied issues. We work in the multiparameter setting so that our main operator is the \emph{strong maximal operator}
\[
\Ms f(x)\coloneqq \sup_{x\in R}\frac{1}{|R|}\int_{R} |f(y)|dy,\quad x\in \R^n,
\]
where the supremum is taken over all rectangular parallelepipeds $R\subseteq \R^n$ with sides parallel to the coordinate axes. This operator is in many senses a prototype for multiparameter harmonic analysis  as it is a geometric maximal operator that commutes with the full $n$-parameter group of dilations $(x_1, x_2, \ldots, x_n) \rightarrow (\delta_1x_1, \delta_2 x_2, \ldots, \delta_n x_n)$.   Unlike the Hardy-Littlewood maximal operator, the strong maximal operator is not of weak type $(1,1)$. It does however satisfy a weak distributional estimate of the form
\[
|\{x\in\R^n:\, \Ms f(x)>\lambda\}|\lesssim_n \int_{\R^n} \frac{|f(x)|}{\lambda}\Big(1+\big(\log^+\frac{|f(x)|}{\lambda}\big)^{n-1}\Big)dx,\quad \lambda>0;
\]
here $\log^+ t\coloneqq \max( \log t,0)$. This endpoint distributional inequality essentially goes back to Jessen, Marcinkiewicz, and Zygmund, \cite{JMZ}, and it allows us to show that the collection of all rectangles in $\R^n$ with sides parallel to the coordinate axes differentiates functions that are locally in $L(\log L)^{n-1} (\R ^n)$. See also \cite{CF} for a geometric proof of the same result.

In this paper we take up the study of weighted analogues of Solyanik estimates for the \emph{sharp Tauberian constants} associated with the basis of axes parallel rectangles.    Recall that, in the unweighted case, the sharp Tauberian constant associated with $\Ms$ is defined by
\[
 \mathsf C_{\mathsf S}(\alpha)\coloneqq \sup_{\substack {E\subset \R^n \\ 0<|E|<+\infty}}\frac{1}{|E|}\abs{\{x\in\R^n:\, \Ms(\ind_E)(x)>\alpha\}},\quad \alpha\in(0,1).
\]
Solyanik showed in \cite{solyanik} that $\mathsf C_{\mathsf S}(\alpha)-1\eqsim_n (1-\alpha)^\frac{1}{n}$ as $\alpha\to 1^-$ and thus we refer to such an asymptotic estimate as a \emph{Solyanik estimate}. Solyanik also showed in \cite{solyanik} an identical estimate for the Hardy-Littlewood maximal  operator defined with respect to cubes with sides parallel to the coordinate axes while in \cite{hp1} a similar estimate is proved for the Hardy-Littlewood maximal operator defined with respect to Euclidean balls.

We recall here that for $\alpha\in(1,\infty)$ the function $\phi_{\mathsf S}(\alpha)\coloneqq \mathsf{C}_{\mathsf{S}}(1/\alpha)$ is the so-called \emph{halo function} of the basis of rectangular parallelepipeds in $\R^n$ with sides parallel to the coordinate axes; by convention we define $\phi_{\mathsf S}(\alpha)\coloneqq \alpha$ for $\alpha\in[0,1]$. More generally, given any collection $\mathcal B$ consisting of bounded open sets in $\R^n$ one can define the halo function $\phi_{\mathcal B}$ with respect to the geometric maximal operator $\mathsf M_{\mathcal{B}}$ defined by
\[
 \mathsf M_{\mathcal B} f(x)\coloneqq \sup_{x\in B\in\mathcal B} \frac{1}{|B|}\int_B |f(y)|dy,\quad x\in\bigcup_{B\in\mathcal B} B,
\]
and $\mathsf M_{\mathcal B} f(x)\coloneqq 0$ otherwise. This definition of $\phi_\mathcal B$ is related to the halo conjecture which claims that the \emph{differentiation basis} $\mathcal B$ should differentiate functions $f$ for which $\phi_{\mathcal B}(f) \in L^1 _{\mathsf {loc}}$; see for example \cites{Guzdif} for an extensive discussion related to the halo problem. Some partial results towards this direction are contained in \cites{Guzdif,hs, Hayes, sjolinsoria ,soria}. Our original goal when studying the sharp Tauberian constants of differentiation bases was to enrich the limited information we have for the corresponding halo functions and, in particular, to provide some continuity and regularity estimates.

The endpoint continuity question as $\alpha\to 1^-$ seems however to relate to a variety of different questions in analysis.   For example, we will see in the current paper that Solyanik estimates also find very concrete applications in the theory of weighted norm inequalities. Indeed, the most important example is Theorem~\ref{t.weightedsol} which shows that \emph{weighted Solyanik estimates} give an alternative characterization of the class of multiparameter Muckenhoupt weights $A_\infty ^*$. In a similar note, one can show quantitatively sharp reverse H\"older inequalities for $A_\infty ^*$ weights assuming some weighted Solyanik estimate and quantitative embeddings of the class of multiparameter Muckenhoupt weights $A_\infty ^*$ into $A_p ^*$. On the other hand, Solyanik estimates, in the unweighted or weighted setting, are intimately related to covering properties of the collections of sets used to define $\mathsf M_{\mathcal B}$, and thus also $\mathsf C_{\mathcal B}$. This is especially relevant when one wants to quantify covering arguments of C\'ordoba-Fefferman type, as in \cite{CF}. See \S\ref{s.apps} for a detailed discussion of these applications of weighted Solyanik estimates.

Recently, Michael Lacey brought to our attention that Solyanik estimates have been implicitly used in a number of papers in multiparameter harmonic analysis; for example, in  \cite{CLMP}, Solyanik estimates for the basis of rectangles are used in order to provide versions of Journ\'e's lemma with \emph{small enlargement}. Furthermore, in \cites{laceyferg2002,laceyterwilleger}, Solyanik estimates play a role in results providing a characterization of the product \textsf{BMO} space of Chang and Fefferman, in terms of commutators. See also \cite{DaPe} for more general results of this type . From recent developments it has become apparent that Solyanik estimates and weighted Solyanik estimates will have a role to play, especially towards the direction of providing quantitative covering arguments in the multiparameter setting, where the one parameter covering arguments of Vitali or Besicovitch type fail. 

Very relevant to the theme of this paper are the \emph{weighted Solyanik estimates} and the \emph{Solyanik estimates with respect to weights}, studied in \cite{hp2} for the case of one-parameter operators. The main purpose of this paper is to prove Solyanik estimates under the presence of weights for the strong maximal operator. In order to explain the terminology, a weighted Solyanik estimate vaguely corresponds to the bound $\Ms:L^p(w)\to L^p(w)$ where the Lebesgue measure in the ambient space is replaced by $w$ but the maximal operator is still defined with respect to the Lebesgue measure. On the other hand, a Solyanik estimate with respect to a weight corresponds to a bound $\Mw:L^p(w)\to L^p(w)$ where the Lebesgue measure is replaced by $w$ both in the ambient space as well as in the definition of the maximal operator.

In this paper we shall see that Solyanik estimates also find very concrete applications in the theory of weighted norm inequalities. In particular we discuss in \S\ref{s.apps} a series of corollaries of weighted multiparameter Solyanik estimates that exhibit an intimate connection to reverse H\"older inequalities, weighted covering lemmas for rectangles in $\R^n$, as well as quantitative embeddings of the class of multiparameter Muckenhoupt weights $A_\infty ^*$ into $A_p ^*$.

\subsection*{Weighted multiparameter Solyanik estimates}In the study of Solyanik estimates in \cite{hp2} the class of Muckenhoupt weights $A_\infty$ comes up naturally as a certain weighted Solyanik estimate for the Hardy-Littlewood maximal operator is shown to actually characterize the class $A_\infty$. It is thus no surprise that the class of strong Muckenhoupt weights $A_\infty ^*$ is central in the current paper. Our approach heavily depends on one-dimensional notions  so we immediately recall the definition of $A_p$ weights on the real line.
\begin{definition}\label{d.ap} We say that a non-negative, locally integrable function $w$ in $\R$, that is, a \emph{weight}, belongs to the \emph{Muckenhoupt class} $A_p$ on the real line, $1<p<+\infty$, if
\[
	[w]_{A_p }\coloneqq \sup_{I}\bigg(	\frac{1}{|I|}\int w(y)dy\bigg)\bigg(\frac{1}{|I|}\int_B w(y)^{-\frac{1}{p-1}}dy\bigg)^{p-1}<+\infty.
\]
where the supremum is taken over all bounded intervals $I\subseteq \R$. The class $A_1$ is defined to be the set of weights $w$ on the real line such that
\[
[w]_{A_1 }\coloneqq	\sup_{I}\bigg(\frac{1}{|I|}\int_I w(y)dy \bigg) \esssup_{I} (w^{-1})<+\infty\;.
\]
Also, we define the class $A_\infty$ to be the set of weights $w$ such that
\[
 [w]_{A_\infty}\coloneqq \sup_I \frac{1}{w(I)}\int_I \MM(w\ind_I)<+\infty.
\]
\end{definition}
Some remarks are in order. Firstly, the class $A_\infty$ can be also described as $A_\infty= \cup_{p>1} A_p $, while many equivalent definitions exist in the literature; see \cite{DMRO}. Definition~\ref{d.ap} for $p=\infty$ goes back to Fujii \cite{Fu}, and Wilson, \cites{W1,W2}. Recently several papers used the Fujii-Wilson constant above in order to provide sharp quantitative weighted bounds for maximal functions and singular integrals; see for example \cites{HytP,HytPR,LM}. We also recall that the class of Muckenhoupt weights $A_p $ characterizes the boundedness property $\MM:L^p(\R,w)\to L^p(\R,w)$ for $p\in(1,\infty)$ where $\MM$ denotes the non-centered Hardy-Littlewood maximal operator on $\R$.

These definitions extend in higher dimensions in different ways. If we replace intervals by cubes in $\R^n$ with sides parallel to the coordinate axes we get the one-parameter Muckenhoupt classes in $\R^n$ which are still denoted by $A_p$. The classes $A_p$ characterize the boundedness of the $n$-dimensional Hardy-Littlewood maximal operator on $L^p(\R^n,w)$. However, if we replace the intervals in Definition~\ref{d.ap} by rectangular parallelepipeds in $\R^n$ with sides parallel to the coordinate axes, the resulting classes define the \emph{strong} or \emph{multiparameter Muckenhoupt weights}, denoted by $A_p ^*$. The class of strong Muckenhoupt weights characterizes the boundedness property $\Ms:L^p(\R^n,w)\to L^p(\R^n,w)$ for $p\in(1,\infty)$ and thus is very relevant to the content of this paper. See for example \cite{GaRu} for a more detailed discussion on these issues.

Here we adopt a one-dimensional point of view on strong Muckenhoupt weights and their corresponding constants. For $x=(x_1,\ldots,x_n)\in \R^n$ let us define the $(n-1)$-dimensional vector $\bar x^j\coloneqq(x_1,\ldots,x_{j-1},x_{j+1},\ldots x_n)\in\R^{n-1}$. We then consider the one-dimensional weight
\[
 w_{\bar x ^j}(t)\coloneqq w(x_1,\ldots,x_{j-1},t,x_{j+1},\ldots x_n),\quad t\in\R.
\]
It is well known that $w\in A_p ^*$ if and only if $w_{\bar x^j} \in A_p$ on the real line, uniformly for a.e. $\bar x^j\in\R^{n-1}$; see \cite{GaRu} or \cite{BaKu}*{Lemma 1.2}. This motivates the following definition.
\begin{definition}\label{d.ainfty} Let $w\in A_p ^*$ be a strong Muckenhoupt weight in $\R^n$ and let $1\leq p \leq \infty$. We define
\[
[w]_{A_p ^*} \coloneqq \sup_{1\leq j\leq n} \esssup_{\bar x^j \in \R^{n-1}} [w_{\bar x ^j}]_{A_p  }.
\]
\end{definition}
The discussion above is then translated to the statement that for $p\in[1,\infty]$ we have that $w\in A_p ^*\Leftrightarrow [w]_{A_p ^*}<+\infty$. We will overview the basic properties of strong Muckenhoupt weights in more detail in \S\ref{s.weights}.

Under the presence of a weight in the ambient space, the natural definition for the sharp Tauberian constant becomes
\[
 \Cs(\alpha)\coloneqq \sup_{\substack {E\subset \R^n \\ 0<w(E)<+\infty}}\frac{1}{w(E)}w(\{x\in\R^n:\, \Ms(\ind_E)(x)>\alpha\}).
\]
Our first main theorem gives a new characterization of the class $A_\infty ^*$ in terms of weighted Solyanik estimates for $\Ms$.
\begin{theorem}\label{t.weightedsol} Let $w$ be a non-negative, locally integrable function in $\R^n$. If $w\in A_\infty ^*$ we have
	\[
	\Cs(\alpha)-1\lesssim_n (1-\alpha)^{(c n[w]_{A_\infty ^*})^{-1}}\quad \text{for all}\quad 1>\alpha>1-e^{-cn[w]_{A_\infty ^*}},
	\]
where $c>0$ is a numerical constant. Furthermore this estimate is sharp in the following sense: if there exist $B,\beta>1$ and $\gamma>0$ such that $\Cs(\alpha)-1\leq B(1-\alpha)^\frac{1}{\beta}$ for all $1 > \alpha>1-e^{-\gamma}$ then $w\in A_\infty ^*$ and $[w]_{A_\infty ^*} \lesssim \beta(1+\max(\gamma/\beta,\ln B))$.
\end{theorem}
It is well known that $A_\infty ^*$ weights satisfy reverse H\"older inequalities. Sharp quantitative versions of these inequalities are contained in several places in the literature as for example in \cites{HytP,HytPR} and \cite{Kin} for the one-parameter case, and in \cites{kint,Lu} for the multiparameter case. In one dimension even more precise results are known which also describe the optimal numerical constants involved in the estimates; see for example \cite{DW} and \cite{V}. As a corollary of Theorem~\ref{t.weightedsol} we obtain a reverse H\"older inequality for strong Muckenhoupt weights.

\begin{theorem}\label{t.rhi} Let $w\in A_\infty ^*$ be a strong Muckenhoupt weight on $\R^n$ and define $[w]_{A_\infty ^*}$ as above. There exists a numerical constant $c>0$ such that
\[
 \Big(\frac{1}{|R|}\int_R w^r \Big)^\frac{1}{r} \lesssim_n \frac{1}{(1-(r-1)(cn[w]_{A_\infty ^*} -1))^{\frac{1}{r}}}\frac{1}{|R|}\int_R w
\]
for all $r<1+\frac{1}{cn[w]_{A_\infty ^*}-1}$. Furthermore, the exponent in the reverse H\"older inequality is optimal up to dimensional constants: if a weight $w$ satisfies
\[
 \Big(\frac{1}{|R|}\int_R w^r \Big)^\frac{1}{r} \leq B \frac{1}{|R|}\int_R w
\]
for all rectangular parallelepipeds $R$ then $w\in A_\infty ^*$ and $[w]_{A_\infty ^*}\lesssim r' (1+\ln B).$
\end{theorem}
It is of some importance to note that the reverse H\"older inequality above holds with an exponent defined with respect to the $A_\infty ^*$-constant from Definition~\ref{d.ainfty}, which is essentially one-dimensional. This results to a wider range for the exponent in the reverse H\"older inequalities for multiparameter weights, compared to the ones that were known or implicit in the literature; indeed, these involve the $A_p ^\mathsf{rec}$-constants which are defined with respect to rectangles and are in general larger than the $A_p ^*$-constants we use here; see \S\ref{s.rhi}.
\subsection*{Multiparameter Solyanik estimates with respect to weights} A parallel investigation concerns the weighted strong maximal operator defined for a non-negative locally integrable function $w$ on $\R^n$ as
\[
\Mw f(x)\coloneqq \sup_{x\in R} \frac{1}{w(R)}\int_R |f(y)|w(y)dy,\quad x\in\R^n.
\]
Of course the same definition makes perfect sense for essentially any locally finite Borel measure $\mu$ in place of $w$. However, our understanding of multiparameter maximal operators defined with respect to measures is rather rudimentary and the case $d\mu(x)=w(x)dx$ for $w\in A_\infty ^*$ is one of the few examples where we have a more or less complete picture of the available bounds. For example it is known that if $w\in A_\infty ^*$ then $\Mw$ is bounded on $L^p(w)$ for $p\in(1,\infty)$; see for example \cite{F}. Surprisingly, the question whether this basic mapping property persists for the case of product doubling measures remains open. See however \cite{hlp} for a related discussion and a characterization of this property in terms of Tauberian conditions.

For a non-negative, locally integrable function $w$ on $\R^n$ we define the sharp Tauberian constant corresponding to $\Mw$ as
\[
 \Cw(\alpha)  \coloneqq \sup_{\substack {E\subset \R^n \\ 0<w(E)<+\infty}}\frac{1}{w(E)}w(\{x\in\R^n:\, \Mw(\ind_E)(x)>\alpha\}).
\]
The second main result of this paper is a Solyanik estimate for $\Cw$ in the case that $w\in A_\infty ^*$.
\begin{theorem}\label{t.solwrtw} Let $w\in A_\infty ^*$ be a strong Muckenhoupt weight. There exists a constant $c_{w,n}>0$ depending only upon $w$ and the dimension such that
	\[
	\Cw(\alpha)-1\lesssim_{w,n} (1-\alpha)^{c_{w,n}}
	\]
as $\alpha\to 1^-$.
\end{theorem}

\section*{Notation} We use the letters $C,c>0$ to denote numerical constants that can change even in the same line of text. A dependence of some constant $c$ on some parameter $\tau$ is indicated by writing $c_\tau$. We write $A\lesssim B$ whenever $A\leq cB$ and $A\eqsim B$ whenever $A\lesssim B$ and $B\lesssim A$. We denote dependencies on parameters by writing, for example, $A\lesssim_\tau B$. A \emph{weight} $w$ is a non-negative locally integrable function and we use the notation $w(E)\coloneqq \int_E w(x)dx$ for measurable sets $E\subseteq \R^n$. Finally we use the letters $R,S$ to denote rectangular parallelepipeds in $\R^n$, which we will frequently colloquially refer to as \emph{rectangles}, whose sides are parallel to the coordinate axes. In the one-dimensional case,  bounded subintervals of the real line are denoted by $I$.

\section{Preliminaries; some properties of \texorpdfstring{$A_\infty ^*$}{Ainfty} weights} \label{s.weights}
The literature concerning one-parameter Muckenhoupt weights is extremely rich and refined, providing very sharp estimates and alternative proofs for all the  properties of interest. In the multiparameter case the literature is quite limited. In many cases, the properties of one-parameter weights extend without difficulty to the multiparameter case. See for example \cite{GaRu} where most of these classical properties of strong Muckenhoupt weights are described. Some attention should be given however when transferring properties from the one-parameter case to the multiparameter case, especially when the endpoint bounds for the corresponding maximal operators are involved. In this section we gather the properties of strong Muckenhoupt weights that we need in the rest of the paper and briefly review their proofs in the multiparameter case.

Concerning the gauges used for strong Muckenhoupt weights, it is quite common in the literature to define $A_p ^*$-constants with respect to rectangles by
\[
 [w]_{A_p ^\mathsf{rec}}\coloneqq \sup_R \Big(\frac{1}{|R|}\int_R w\Big) \Big(\frac{1}{|R|}\int_R w^{-\frac{1}{p-1}}\Big)^{p-1}
\]
with the supremum taken over all rectangular parallelepipeds in $\R^n$ with sides parallel to the coordinate axes. In fact, there are not so many, if any, quantitative weighted bounds for multiparameter weights in the literature; see however \cite{Lu} and the references therein.  The following simple lemma gives the equivalence of the definition of $A_p ^*$ in terms of the constants $[w]_{A_p ^*}$ and $[w]_{A_p ^\mathsf{rec}}$. Note however the qualitative flavor of the statement of the lemma in one of the two directions.
\begin{lemma}\label{l.dir} Let $w$ be a non-negative, locally integrable function in $\R^n$ and let $1\leq p \leq	 \infty$. Then $w\in A_p ^*$ if and only if $[w]_{A_p ^*}<+\infty$ if and only if $[w]_{A_p ^\mathsf{rec}}<+\infty$. Furthermore, for all $p\in[1,\infty)$ we have $[w]_{A_p ^*}\leq [w]_{A_p ^\mathsf{rec}}$.
\end{lemma}
This lemma is classical and the proof can be found for example in \cite{GaRu}. The inequality relating the two constants above is a simple consequence of the Lebesgue differentiation theorem. Observe that in dimension one there is no distinction between one-parameter and multiparameter weights so we will just use the notation $[w]_{A_p}$ for one-dimensional weights.

Observe that in the lemma above, equality may occur in $[w]_{A_p ^*}\leq [w]_{A_p ^\mathsf{rec}}$ as for example in the case $w(x)\coloneqq \upsilon(x_1)$, where $\upsilon$ is a one-dimensional weight. However, $[w]_{A_p ^*}$ can be a lot smaller than $[w]_{A_p ^\mathsf{rec}}$ as for example in the case $w(x)\coloneqq \upsilon(x_1)\cdots\upsilon(x_n)$ with $\upsilon$ as above. Indeed, in this case we have $[w]_{A_p ^*}=[\upsilon]_{A_p}$ while $[w]_{A_p ^\mathsf{rec}}=[\upsilon]_{A_p} ^n \gg [w]_{A_p ^*} $.

We now recall one of the most important properties of Muckenhoupt weights, the fact that they satisfy a reverse H\"older inequality, together with an alternative characterization of $A_\infty ^*$. We state here a quantitative one-dimensional version which is tailored to the needs of this paper.
\begin{lemma}\label{l.rhi} Let $w$ be a non-negative, locally integrable function on the real line.
\begin{itemize}
 \item [(i)] If $w\in A_\infty$ then for all intervals $I\subseteq \R$ and all $0<\epsilon \leq (4 [w]_{A_\infty }-1)^{-1}$ we have the reverse H\"older inequality
\[
 \frac{1}{|I|}\int_I w^{1+\epsilon} \leq 2\Big(\frac{1}{|I|} \int_I w\Big)^{1+\epsilon}.
\]
Furthermore, for all intervals $I\subseteq \R$ and all measurable $E\subseteq I$ we have
\[
\frac{w(E)}{w(I)}\leq 2 \Big(\frac{|E|}{|I|}\Big) ^{(4[w]_{A_\infty })^{-1}}.
\]
\item[(ii)] Conversely, if there exist constants $B,\beta \geq 1$ such that for all intervals $I\subseteq \R$ and all measurable $E\subseteq I$ we have
\[
 \frac{w(E)}{w(I)}\leq B \Big(\frac{|E|}{|I|}\Big)^\frac{1}{\beta}
\]
then $w$ satisfies the reverse H\"older inequality
\[
 \Big(\frac{1}{|I|}\int_I w^r \Big)^\frac{1}{r} \leq B^\frac{\beta}{r'} \Big(\frac{ \beta'-1}{\beta'-r}\Big)^\frac{1}{r} \frac{1}{|I|}\int_I w
\]
for all intervals $I\subseteq \R$, whenever $r<\beta'$; here $1/\beta+1/\beta'=1$. Furthermore $w\in A_\infty$ and $[w]_{A_\infty} \lesssim \beta (1+\ln B)$.
\end{itemize}
\end{lemma}

\begin{proof}The reverse H\"older inequality of (i) is the one-dimensional case of \cite{HytPR}*{Theorem 2.3}. The second statement in (i) follows immediately by a simple application of H\"older's inequality and the reverse H\"older inequality. For (ii) let us fix an interval $I$. In order to prove the reverse H\"older inequality in the statement of the lemma we can assume that $w(I)/|I|=1$. Defining $E_\lambda\coloneqq \{x\in I:\, w>\lambda\}$ we then have
\[
 \frac{|E_\lambda|}{|I|} \leq \frac{1}{\lambda} \frac{w(E_\lambda)}{w(I)} \leq \frac{B}{\lambda} \Big(\frac{|E_\lambda|}{|I|}\Big)^\frac{1}{\beta},
\]
where the first inequality is trivial and the second inequality following by the hypothesis of (ii). Thus for $\lambda>0$ we get the estimate
\[
 \frac{|E_\lambda|}{|I|}\leq \Big(\frac{B}{\lambda}\Big)^{\beta'}.
\]
Using the hypothesis of (ii) the previous inequality implies
\[
 \frac{w(E_\lambda)}{w(I)} \leq B^{\beta'} \lambda^{-\frac{\beta'}{\beta} }\;.
\]
Now for $1<r<1+\beta'/\beta=\beta'$ we can estimate
\[
 \begin{split}
 \frac{1}{w(I)} \int_I w^r& =\frac{1}{w(I)}\int_I w^{r-1} w= \frac{1}{w(I)}\int_0 ^\infty (r-1)\lambda^{r-2} w(E_\lambda)d\lambda
\\
&\leq B^{\beta(r-1)} + (r-1)B^{\beta'}\frac{B^{\beta(r-1-\beta'/\beta)}}{\beta'/\beta -(r-1)} =B^{\beta(r-1)}\frac{\beta'-1}{\beta'-r}
\end{split}
\]
which is the desired reverse H\"older inequality. In order to see the estimate for $[w]_{A_\infty}$ we utilize the $L^p$ bounds of the Hardy-Littlewood maximal operator  $\MM$ on the real line. As in the proof of \cite{HytP}*{Theorem 2.3} we have for every interval $I\subseteq \R$ that
\[
\begin{split}
 \frac{1}{|I|}\int_I \MM(w\ind_I)& \leq \Big(\frac{1}{|I|}\int_I \big(\MM(w\ind_I)\big)^r\Big)^\frac{1}{r}\lesssim r' \Big(\frac{1}{|I|}\int_I   w ^r\Big)^\frac{1}{r}
\\
& \leq r' B^\frac{\beta}{r'} \Big(\frac{ \beta'-1}{\beta'-r}\Big)^\frac{1}{r} \frac{1}{|I|}\int_I w.
\end{split}
\]
Taking the supremum over all intervals $I$ and using the Fujii-Wilson definition of $[w]_{A_\infty}$ we get
\begin{equation}\label{e.ainfest}
 [w]_{A_\infty} \lesssim \inf_{1<r<\beta'} r' B^\frac{\beta}{r'} \Big(\frac{ \beta'-1}{\beta'-r}\Big)^\frac{1}{r}.
\end{equation}
If $\beta>2$ then consider $r_o\coloneqq1+  \beta ' /(2\beta(1+\ln B))$. Obviously $1<r_o<1+\beta'/\beta=\beta'$ and we  can estimate
\[
  r_o ' \eqsim \beta(1+\ln B)  \quad\text{and}\quad  B^\frac{\beta}{r_o '} \Big(\frac{ \beta'-1}{\beta'-r_o}\Big)^\frac{1}{r_o}\lesssim 1 .
\]
By \eqref{e.ainfest} this gives the claim for $\beta>2$. If $\beta \in [1, 2] $ then the hypothesis is always true for $\beta=2$ so the previous argument gives
\[
 [w]_{A_\infty} \lesssim (1+\ln B) \lesssim \beta (1+\ln B)
\]
and we are done.
\end{proof}
We close this section with a technical lemma which will be useful in a number of occasions when one assumes, or manages to prove, some Solyanik estimate for $\Cs$.  In particular, we will use this lemma when showing the optimality of Theorem~\ref{t.weightedsol} as well as in the proof of Theorem~\ref{t.rhi}

\begin{lemma}\label{l.fromsol} Let $w$ be a weight in $\R^n$ and assume that there exist constants $B,\beta\geq 1$ and $\gamma>0$ such that
\[
\Cs(\alpha)-1 \leq B (1-\alpha)^\frac{1}{\beta}\quad\text{for all}\quad 1>\alpha >1-e^{-\gamma}\;.
\]
Then the following hold:
\begin{itemize}
 \item [(i)] For all rectangular parallelepipeds $R\subseteq\R^n$ and all measurable sets $E\subseteq R$ we have
\[
 \frac{w(E)}{w(R)} \leq \max(B,e^\frac{\gamma}{\beta}) \Big(\frac{|E|}{|R|}\Big)^\frac{1}{\beta},
\]
\item[(ii)] For all rectangular parallelepipeds $R\subseteq R^n$ we have the reverse H\"older inequality
\[
 \Big(\frac{1}{|R|} \int_R w^r \Big)^\frac{1}{r} \leq \max(B,e^\frac{\gamma}{\beta})^\frac{\beta}{r'} \Big(\frac{ \beta'-1}{\beta'-r}\Big)^\frac{1}{r} \frac{1}{|R|}\int_R w
\]
for all $1<r<\beta'$.
\item[(iii)] We have that $w\in A_\infty ^*$ and $[w]_{A_\infty ^*} \lesssim \beta (1+\max(\gamma/\beta,\ln B))$.
\end{itemize}
\end{lemma}

\begin{proof} For (i),  let $\alpha\in(0,1)$ with $\alpha>1-e^{-\gamma}$ and consider a rectangular parallelepiped $R\subseteq \R^n$ with sides parallel to the coordinate axes and a measurable set $S\subseteq R$. If $|S|/|R|>\alpha$ then $R\subseteq \{x\in\R^n:\, \Ms(\ind_S)(x)>\alpha\}$. Thus $w(R)\leq \Cs(\alpha) w(S)$ and calling $E\coloneqq R\setminus S$ we have
\[
 w(E)\leq \frac{\Cs(\alpha)-1}{\Cs(\alpha)}w(R)\leq B(1-\alpha)^\frac{1}{\beta} w(R)\quad \text{whenever}\quad \frac{|E|}{|R|}<1-\alpha,\quad \alpha>1-e^{-\gamma},
\]
by the hypothesis and the fact that we always have $\Cs(\alpha)\geq 1$. Letting $\alpha\to 1-|E|/|R|$ we get
\[
 \frac{w(E)}{w(R)}\leq B\Big(\frac{|E|}{|R|}\Big)^\frac{1}{\beta} \quad \text{whenever}\quad \frac{|E|}{|R|}<e^{-\gamma}.
\]
If $|E|/|R|>e^{-\gamma}$ we trivially have
\[
 \frac{w(E)}{w(R)}\leq e^\frac{\gamma}{\beta}\Big(\frac{|E|}{|R|}\Big)^\frac{1}{\beta}.
\]
Thus for every rectangular parallelepiped $R$ and measurable $E\subseteq R$ we can conclude
\[
 \frac{w(E)}{w(R)}\leq \tilde B \Big(\frac{|E|}{|R|}\Big)^\frac{1}{\beta} \quad \text{with}\quad \tilde B\coloneqq \max (B,e^\frac{\gamma}{\beta})
\]
as we wanted.

The proof of (ii) is identical to the proof of the reverse H\"older inequality in (ii) of Lemma~\ref{l.rhi}.

For (iii) we begin by fixing some $j\in\{1,\ldots,n\}$. For a.e. $\bar x^j\in\R^{n-1}$, the estimate in (i) and the Lebesgue differentiation theorem implies that for all intervals $I\subseteq\R$ and all measurable sets $E\subseteq I$ we have
\[
 \frac{w_{\bar x^j}(E)}{w_{\bar x^j}(I)}\leq \max(B,e^\frac{\gamma}{\beta}) \Big(\frac{|E|}{|I|}\Big)^\frac{1}{\beta}.
\]
By (ii) of Lemma~\ref{l.rhi} this implies that for a.e. $\bar x^j\in\R^{n-1}$ we have $w_{\bar x^j}\in A_\infty$ and $[w_{\bar x^j}]_{A_\infty} \lesssim \beta(1+\max(\gamma/\beta,ln B)$. Since the previous estimate is uniform in $j\in\{1,\ldots,n\}$ and a.e. $\bar x^j\in\R^{n-1}$ it follows that $[w]_{A_\infty ^*} \lesssim  \beta(1+\max(\gamma/\beta,\ln B)$ as we wanted.
\end{proof}

\section{One-dimensional results}
A typical technique of proof in  multiparameter harmonic analysis is based on induction or \emph{reduction of parameters}. The base step of the induction is the one-parameter case which is naturally, but not necessarily, identified with the one-dimensional case. Thus we gather here all the weighted one-dimensional results which will be used in the inductive proofs in the rest of the paper. The unweighted versions of these results are contained in \cite{solyanik}. Here we adopt a slight variation introduced in \cite{hp1} which is more suitable for our purposes.

\subsection{Weighted one-dimensional Solyanik estimates} Remember that the one-dimensional Hardy-Littlewood maximal operator is defined by
\[
\mathsf M_1 f(x) \coloneqq\sup_{x\in I} \frac{1}{|I|}\int_I |f(y)|dy,\quad x\in \R,
\]
where the supremum is taken over all intervals $I\subseteq \R$ with $I\ni x$. The corresponding sharp Tauberian constant with respect to a weight $w$ is defined for $\alpha\in(0,1)$ as
\[
\CC(\alpha)\coloneqq \sup_{\substack {E\subset \R \\ 0<w(E)<+\infty}}\frac{1}{w(E)}w(\{x\in\R:\, \mathsf M_1 (\ind_E)(x)>\alpha\}).
\]
The following lemma is the weighted version of a completely analogous lemma from \cite{hp1}.
\begin{lemma} \label{l.weightsol1d} Let $w\in A_\infty $ be a Muckenhoupt weight on the real line and let $E\subset \R$ be a measurable set with $0<|E|<+\infty$. Then for all $0\leq \gamma<\alpha<1$ with $1-\alpha<4^{-4[w]_{A_\infty }}(1-\gamma)$ we have
	\[
	w(\{x\in\R:\, \MM(\ind_E+\gamma\ind_{E^\mathsf{c}})>\alpha\})\leq \Big(1-4 \Big(\frac{1-\alpha}{1-\gamma}\Big)^{(4[w]_{A_\infty })^{-1}}\Big)^{-1}w(E).
	\]
\end{lemma}

\begin{proof} For convenience we set $f_{E,\gamma}\coloneqq \ind_E +\gamma \ind_{E^\mathsf{c}}$ and first prove the case $\gamma>0$. There exists a countable collection of intervals $\{\tilde I_j\}_j$ such that $E_{\alpha,\gamma}\coloneqq \{x\in\R:\, \MM(\ind_E \gamma+\ind_{E^\mathsf{c}})(x)>\alpha\}\subseteq \cup_j \tilde I_j$ and
	\[
\frac{1}{|\tilde I_j|}	\int_{\tilde I_j} f_{E,\gamma}>\alpha.
	\]
Fixing some compact $K\subseteq E_{\alpha,\gamma}$ we have that $K\subseteq_j \cup_j I_j $ for a finite collection $\{I_j\}_j\subseteq \{\tilde I_j\}_j$. Furthermore, there exists a subcollection $\{I_{j_k}\}_k\subseteq \{I_j\}_j$ such that $\cup_k I_{j_k} = \cup _j I_{j}$ and $\sum_j \ind_{I_{j_k}}\leq 2$; see for example \cite{Gar}*{p. 24} for more details on this classical covering argument. Observe that for each $k$ we then have
\[
\frac{1}{|I_{j_k}|}\int_{I_{j_k}}\ind_E>\frac{\alpha-\gamma}{1-\gamma}
\]
and so $|I_{j_k}\cap E^{\mathsf{c}}|/|I_{j_k}| \leq  (1-\alpha) /(1-\gamma)$. Using (i) of Lemma~\ref{l.rhi} for $w$ we get that
\[
\frac{w(I_{j_k}\cap E^{\mathsf{c}})}{w(I_{j_k})}\leq 2 \Big(\frac{1-\alpha}{1-\gamma}\Big)^{(4[w]_{A_\infty })^{-1}}.
\]
We thus have
\[
\begin{split}
w\big(\bigcup_k I_{j_k}\big)&\leq w(E)+2 \Big(\frac{1-\alpha}{1-\gamma}\Big)^{(4[w]_{A_\infty })^{-1}} \sum_j w(I_{j_k})
\\
&\leq   w(E)+4 \Big(\frac{1-\alpha}{1-\gamma}\Big)^{(4[w]_{A_\infty })^{-1}} w\big(\bigcup_k I_{j_k}\big)
\end{split}
\]
and accordingly
\[
w(K)\leq \Big(1- 4 \Big(\frac{1-\alpha}{1-\gamma}\Big)^{(4[w]_{A_\infty })^{-1}}\Big)^{-1}w(E).
\]
This easily implies the desired estimate for $\gamma>0$.

Now for $\alpha>1-4^{-4[w]_{A_\infty}}$ we have for sufficiently small $\delta>0$
\[
w(\{x\in\R:\, \MM(\ind_E)(x)>\alpha\})\leq \Big(1- 4 \Big(\frac{1-\alpha}{1-\delta}\Big)^{(4[w]_{A_\infty })^{-1}}\Big)^{-1}w(E).
\]
Letting $\delta\to 0^+$ we get the claim for $\gamma=0$ as well.
\end{proof}

\subsection{One-dimensional Solyanik estimates with respect to Borel measures} In this section we consider the weighted maximal operator in one dimension
\[
\MMw f(x)\coloneqq \sup_{x\in I}\frac{1}{\mu(I)}\int_I |f(y)|d\mu(y),\quad x\in\R,
\]
where the supremum is taken with respect to all intervals $I\ni x$. The sharp Tauberian constant associated with $\MMw$ is then defined as
\[
\CCw (\alpha)\coloneqq \sup_{\substack {E\subset \R \\ 0<\mu(E)<+\infty}}\frac{1}{\mu(E)}\mu(\{x\in\R:\, \MMw (\ind_E)(x)>\alpha\}).
\]
In this case the corresponding Solyanik estimates are very simple to prove.
\begin{lemma} Let $\mu$ be a non-negative, locally finite Borel measure on the real line. Then there exists a numerical constant $c>0$ such that for all $0\leq \gamma <\alpha<1$ we have
	\[
	\mu(\{x\in\R:\, \MM ^\mu(\ind_E+\gamma\ind_{E^\mathsf{c}})(x)>\alpha\})\leq \Big(1+ 2 \Big(\frac{1-\alpha}{\alpha-\gamma}\Big)\Big)\mu(E).
	\]
\end{lemma}

\begin{proof} It is enough to prove the lemma for $\gamma>0$. As in the proof of Lemma~\ref{l.weightsol1d}, given a compact $K\subseteq \{x\in\R:\, \MMw(\ind_E+\gamma\ind_{E^\mathsf{c}})(x)>\alpha\} $ there exist disjoint intervals $\{I_k\}_k$ with $\sum_k \ind_{I_k}\leq 2$, $K\subseteq \cup_k I_k$, and
	\[
\frac{1}{\mu(I_k)}\int_{I_k} f_{E,\gamma}d\mu >\alpha.
	\]
Observe that then we get $\mu(E\cap I_k)/\mu(I_k)>(\alpha-\gamma)/(1-\gamma)$.	We thus have
\[
\begin{split}
\sum_k \mu( I_k \cap E^\mathsf{c})& \leq\frac{1-\alpha}{1-\gamma} \sum_k \mu(I_k)\leq \frac{1-\alpha}{1-\gamma}\frac{1}{\alpha}\sum_k\int_{I_k} f_{E,\gamma}d\mu
\\
& \leq \frac{1}{\alpha}\frac{1-\alpha}{1-\gamma}\sum_k\big(\mu(E\cap I_k)+\gamma \mu(E^\mathsf{c}\cap I_k)\big)
\\
& \leq \frac{2}{\alpha}\frac{1-\alpha}{1-\gamma}\mu(E)+\frac{\gamma}{\alpha}\frac{1-\alpha}{1-\gamma}\sum_k \mu(I_k \cap E^\mathsf{c}).
\end{split}
\]
Since $0<\gamma<\alpha<1$ we can conclude that
\[
\sum_k \mu(I_k \cap E^\mathsf{c})\leq  \frac{2(1-\alpha)}{\alpha-\gamma} \mu(E),
\]
and thus
\[
\mu(K)\leq \big(1+ \frac{2(1-\alpha)}{\alpha-\gamma} \big) \mu(E).
\]
This proves the desired claim.
\end{proof}
As a corollary we immediately obtain a one-dimensional Solyanik estimate with respect to Borel measures.
\begin{corollary}\label{c.oned} Let $\mu$ be a non-negative locally finite Borel measure on the real line. Then for all $\alpha\in(0,1)$ we have
	\[
	\CCw(\alpha)-1\leq 2\frac{1-\alpha}{\alpha}.
	\]
\end{corollary}
 Observe that the previous corollary is an extension of Theorem~\ref{t.solwrtw} in the one-dimensional case. It is important to note here that the one-dimensional result is uniform over the class of Borel measures, thus strictly stronger that Theorem~\ref{t.solwrtw}. In general, Solyanik estimates do not hold for the Hardy-Littlewood maximal operator $\mathsf M ^\mu$, or the strong maximal operator $\Ms ^\mu$, defined with respect to arbitrary locally finite Borel measures $\mu$ on $\R^n$, whenever $n\geq 2$. A quick example of this type of behavior is given as follows. Let $\{S_j\}_j$ be a countable collection of sets in $\R^n$, $n\geq 2$, all of which contain the origin and such that, for each $j$ there exists $x_j\in S_j\setminus\cup_{k\neq j} S_k$. Then define the locally finite Borel measure $\mu\coloneqq\delta_0+\sum_j c_j \delta_{x_j}$ for a sequence $\{c_j\}_j$ of positive real numbers with $\lim_{j\to +\infty} c_j=0$ and $\sum_j c_j=+\infty$. If the $S_j$'s are cubes this shows that the Hardy-Littlewood maximal operator $\mathsf M^\mu$, defined with respect to $\mu$, does not satisfy any Solyanik estimate. If the $S_j$'s are rectangular parallelepipeds with sides parallel to the coordinate axes the same example shows that $\Ms ^\mu$ does not satisfy any Solyanik estimates either. In particular, these operators are unbounded on $L^p(\mu)$ for all $p<\infty$.

 The discussion above shows that extending  Corollary~\ref{c.oned} to higher dimensions will require some additional hypothesis on $\mu$. For example, the corollary is still true in $\R^n$ uniformly over all Borel measures which are tensor products of one-dimensional Borel measures as above. A less trivial generalization is contained in Theorem~\ref{t.solwrtw} which however is restricted to measures of the form $d\mu(x)=w(x)dx$ for $w\in A_\infty ^*$. On the other hand, one could consider a version of Corollary~\ref{c.oned} for the \emph{centered} Hardy-Littlewood maximal operator, or the \emph{dyadic} maximal operator defined with respect to some locally finite Borel measure $\mu$. In these cases the result easily extends to $\R^n$ and is uniform over all Borel measures as above. This is an easy consequence of the Besicovitch covering theorem, and the Calder\'on-Zygmund decomposition, respectively.

\section{Weighted Solyanik estimates for the strong maximal operator} We now move to the study of weighted Solyanik estimates for strong Muckenhoupt weights in higher dimensions. We actually prove a stronger estimate which we describe below.

Let $B\coloneqq \{\beta_1,\ldots,\beta_N\}$ be an ordered set of indices with each $\beta_j\in\{1,\ldots,n\}$. Note that we allow the case that $\beta_j=\beta_k$ for $j\neq k$. Then we define the maximal operator $\mathsf M_B$ as
\[
\mathsf M_B \coloneqq \mathsf M_{\beta_1}  \cdots \mathsf M_{\beta_N},
\]
that is, $\mathsf M_B$ is the composition of the operators $M_{\beta_N},\ldots,M_{\beta_1}$, where $\mathsf M_j$ denotes the directional maximal operator acting on the $j$-th direction of $\R^n$
\[
\mathsf M_j f(x)\coloneqq \sup_{s<x_j<t}\frac{1}{t-s}\int_s ^t |f(x_1,\ldots,x_{j-1},u,x_{j+1},\ldots,x_n)|du,\quad x\in\R^n.
\]

\begin{lemma}\label{l.iter} Let $w\in A_\infty ^*$ be a strong Muckenhoupt weight on $\R^n$ and $E\subset\R^n$ be a measurable set with $0<|E|<+\infty$. Let $\alpha_1\in(0,1)$ and $B$ as above. For $j\in\{1,\ldots,|B|-1\}$ we define $1-\alpha_{j+1}\coloneqq (1-\alpha_1)(1-\alpha_j)$. Then for all  $\alpha_1>1-4^{-4[w]_{A_\infty ^*}}$ we have
	\[
w(\{x\in\R^n:\,  \mathsf M_B \ind_E(x)>\alpha_{|B|}\})\leq (1- 4 (1-\alpha_1)^{(4[w]_{A_\infty ^*})^{-1}})^{-|B|}  w(E)\;.
	\]
\end{lemma}

\begin{proof} Let us fix a strong Muckenhoupt weight $w\in A_\infty ^*$. For $\bar x^j\in\R^{n-1}$ we remember that the one-dimensional weight $w_{\bar x ^j}$, defined as $w_{\bar x^j}(t)\coloneqq w(x_1,\ldots,x_{j-1},t,x_{j+1},\ldots,x_n)$ for $t\in \R$, is a Muckenhoupt weight in $\R$ uniformly in $\bar x^j$; that is we have $[w_{\bar x^j}]_{A_\infty}  \leq [w]_{A_\infty ^*}$ for all $\bar x^j \in \R^{n-1}$ and $[w]_{A_\infty ^*}$ is as in Definition~\ref{d.ainfty}.

The proof is by way of induction on the size $|B|$. For $|B|=1$ we can assume without loss of generality that $B=\{1\}$; this is just for notational convenience. Let then $\bar x^1\in\R^{n-1}$ be temporarily fixed. Since $\alpha_1>1-4^{-4[w]_{A_\infty ^*}} \geq 1-4^{-4[w_{\bar x^1}]_{A_\infty}}$ we can use Lemma~\ref{l.weightsol1d} in order to estimate
\[
\begin{split}
w_{\bar x^1}(\{t\in \R:\, \mathsf M_1 \ind_E (t,\bar x^1)>\alpha_1\})&\leq (1- 4 (1-\alpha_1)^{(4[w_{\bar x^1}]_{A_\infty })^{-1}})^{-1}w_{\bar x^1}(E)
\\
& \leq (1- 4 (1-\alpha_1)^{(4[w]_{A_\infty ^*})^{-1}})^{-1}w_{\bar x^1}(E).
\end{split}
\]
Integrating over $\bar x^1\in\R^{n-1}$ we get the desired estimate for $|B|=1$.

Suppose now that
\[
w(\{x\in\R^n:\,  \mathsf M_B \ind_E(x)>\alpha_{j} \})\leq (1- 4 (1-\alpha_1)^{(4[w]_{A_\infty ^*})^{-1}})^{-j}  w(E)
\]
for all strong Muckenhoupt weights $w\in A_\infty$ and for all sets of indices $B$ with $|B|=j$. We proceed to show the corresponding the corresponding estimate for all sets of indices $B$ with $|B|=j+1$. Without loss of generality we can assume that $B=\{1,\beta_2,\ldots,\beta_{j+1}\}$. We define

\[
E_j\coloneqq \{x\in\R^n:\, \mathsf M_{\beta_2}\cdots \mathsf M_{\beta_{j+1}} \ind_E(x)>\alpha_j \}.
\]
Observe that
\[
\begin{split}
 \mathsf M _B \ind_E(x)& \leq \mathsf M_1(\ind_{E_j} \mathsf M _{\beta_2} \cdots \mathsf M_{\beta_{j+1}}  \ind_E + \ind_{E_j ^\mathsf{c}}\mathsf M _{\beta_2} \cdots \mathsf M_{\beta_{j+1}} \ind_E)(x)
\\
& \leq \mathsf M_1  (\ind_{E_j}+\alpha_j \ind_{E_j ^\mathsf{c}})(x).
\end{split}
\]

We fix $\bar x^1\in\R^{n-1}$. Since $1-\alpha_{j+1}=(1-\alpha_1)(1-\alpha_j)<4^{-4[w]_{A_\infty ^*}}(1-\alpha_j)$ we can usethe inequality above together with Lemma~\ref{l.weightsol1d} to estimate

\[
\begin{split}
w_{\bar x^1}(\{t\in\R:\, \mathsf M_ B \ind_E (t,\bar x^1)>\alpha_{j+1}\}) &\leq w_{\bar x^1}(\{t\in\R:\, \mathsf M_1  (\ind_{E_j}+\alpha_j \ind_{E_j ^\mathsf{c}}) (t,\bar x^1)>\alpha_{j+1}\})
\\
& \leq \Big(1- 4 \Big(\frac{1-\alpha_{j+1}}{1-\alpha_j}\Big)^{(4[w_{\bar x^1}]_{A_\infty })^{-1}}\Big)^{-1} w_{\bar x^1}(E_j).
\end{split}
\]

Integrating over $\bar x^1\in\R^{n-1}$ and using the inequality $[w_{x'}]_{A_\infty}\leq [w]_{A_\infty ^*}$ we get
\[
\begin{split}
w(\{x\in\R^n:\, \mathsf M_{B}\ind_E>\alpha_{j+1}\})& \leq\Big(1- 4 \Big( \frac{1-\alpha_{j+1}}{1-\alpha_j}\Big)^{(4[w]_{A_\infty ^*})^{-1}}\Big)^{-1} w(E_j)
\\
&= \big(1- 4  ( 1-\alpha_1)^{(4[w]_{A_\infty ^*})^{-1}}\big)^{-1} w(E_j)\;.
\end{split}
\]
The inductive hypothesis now implies that
\[
 w(E_j)\leq  (1- 4 (1-\alpha_1)^{(4[w]_{A_\infty ^*})^{-1}})^{-j} w(E)
\]
which together with the previous estimate completes the inductive proof of the lemma.
\end{proof}
We can now complete the proof of Theorem~\ref{t.weightedsol}.
\begin{proof}[Proof of Theorem~\ref{t.weightedsol}]Assume that $w\in A_\infty ^*$ is a strong Muckenhoupt weight. We use the elementary estimate $\Ms f \leq \mathsf M_1\cdots \mathsf M_n f$ and Lemma~\ref{l.iter} with $B\coloneqq\{1,2,\ldots,n\}$ and $\alpha_1\coloneqq 1-(1-\alpha)^\frac{1}{n}$ to conclude
	\[
	w(\{x\in\R^n:\, \Ms\ind_E(x)>\alpha\})\leq \big(1- 4 (1-\alpha)^{(4n[w]_{A_\infty ^*})^{-1}}\big)^{-n} w(E)
	\]
for $\alpha>1-e^{-4(\ln 4)n [w]_{A_\infty ^*} }$. It follows that
\[
\Cs(\alpha)-1\lesssim_n (1-\alpha)^{ (4 n [w]_{A_\infty ^*})^{-1}} \quad\text{for all}\quad \alpha>1-e^{-4 (\ln 8) n[w]_{A_\infty ^*}}
\]
and the implied constant depends only upon dimension. The optimality part of the theorem follows immediately by Lemma~\ref{l.fromsol}.
\end{proof}

\section{Some applications of weighted Solyanik estimates}\label{s.apps} In this section we present some applications of the multiparameter weighted Solyanik estimates of Theorem~\ref{t.weightedsol}. These show that Solyanik estimates become a very natural and useful tool in the theory of weighted norm inequalities. An underlying principle, which is due to the multiparameter nature of the weights involved, is that we can many times reduce to the problem under study to a one dimensional one and then lift it again to higher dimensions.

\subsection{A reverse H\"older inequality for \texorpdfstring{$A_\infty ^*$}{Ainfty*}}\label{s.rhi} As a corollary of the weighted multiparameter Solyanik estimate we get, rather unexpectedly, a reverse H\"older inequality for multiparameter Muckenhoupt weights. This is the content of Theorem~\ref{t.rhi} which we now prove.
\begin{proof}[Proof of Theorem~\ref{t.rhi}] As $w\in A_\infty ^*$, Theorem~\ref{t.weightedsol} implies that $w$ satisfies the Solyanik estimate
\[
 \Cs(\alpha)-1 \lesssim_n (1-\alpha)^{(cn[w]_{A_\infty ^*})^{-1}}\quad \text{for all}\quad \alpha>1-e^{-c n [w]_{A_\infty ^*}},
\]
where $c>0$ is a numerical and the implied constant depends only on the dimension $n$. Thus Lemma~\ref{l.fromsol} implies that for every rectangular parallelepiped $R\subseteq \R^n$ we have
\[
 \Big(\frac{1}{|R|}\int_R w^r \Big)^\frac{1}{r} \lesssim_n \big(1-(r-1)(cn[w]_{A_\infty ^*} -1)\big)^{-\frac{1}{r}}\frac{1}{|R|}\int_R w
\]
for all $r<1+\frac{1}{cn[w]_{A_\infty ^*}-1}$. The optimality of the exponents up to dimensional constants follows from Lemma~\ref{l.fromsol}.
\end{proof}

The reader may appreciate that the exponent in the reverse H\"older inequality provided in Theorem~\ref{t.rhi}, $1+(cn[w]_{A_\infty ^*}-1)^{-1}$, is in terms of the essentially one-dimensional $A_\infty ^*$-constant from Definition~\ref{d.ainfty}, and represents an improvement over a more typical  reverse H\"older exponent  given in terms of the $A_\infty$-constant $[w]_{A_{\infty}^{\mathsf{rec, H}}}$ associated to the Hru\v{s}\v{c}ev constant, \cite{Hru}, defined by
\[
[w]_{A_{\infty}^{\mathsf{rec,H}}} \coloneqq \lim_{p \rightarrow \infty}[w]_{A_p ^\mathsf{rec}}=\sup_R \Big(\frac{1}{|R|}\int_R w\Big)\exp\Big(\frac{1}{|R|}\int_R \log w^{-1}\Big),
\]
where the supremum is taken over rectangular parallelepipeds in $\R^n$ with sides parallel to the coordinate axes. To see this improvement, let $[\nu]_{A_{\infty}^ \mathsf{H}}$ denote the Hru\v{s}\v{c}ev constant of a weight $\nu$ on $\R^1$; note that for a.e. $\bar x^j \in \R^{n-1}$ we have
\[
[w_{\bar x^j}]_{A_{\infty}^{\mathsf H}} \leq [w]_{ A_\infty  ^\mathsf{rec,H} } .
\]
Furthermore, as was shown in \cite{HytP}, the Fujii-Wilson constant of a weight on $\mathbb{R}^1$ is bounded above by a constant times the Hru\v{s}\v{c}ev constant of the weight. Thus for a.e. $\bar x^j\in\R^{n-1}$ we get $[w_{\bar x^j}]_{A_\infty}\lesssim [w_{\bar x^j}]_{A_\infty ^\mathsf{H}}\leq [w]_{ A_\infty  ^\mathsf{rec,H} } $ so that $[w]_{A_\infty ^*}\lesssim [w]_{A_\infty ^\mathsf{rec,H}}$.

One can argue in a similar fashion and relate Solyanik estimates to reverse H\"older inequalities for $A_p ^*$ weights when $p\in (1,\infty)$. Note however that, while the sharp reverse H\"older inequalities for multiparameter $A_1 ^*$ weights are known from \cite{kint} to hold with exponents and constants independent of the dimension, this can never be captured by Solyanik estimates. Indeed, in the unweighted case we have that $\mathsf{C}_{\mathsf S}(\alpha)-1\eqsim_n (1-\alpha)^\frac{1}{n}$. Thus the dependence on the dimension appearing in the weighted Solyanik estimates of Theorem~\ref{t.weightedsol} is essentially optimal and no dimension free reverse H\"older inequalities can be produced with the methods of this paper.

As another corollary of Theorem~\ref{t.weightedsol} and Lemma~\ref{l.fromsol} we obtain
\begin{corollary}\label{c.homa*} Let $w\in A_\infty ^*$. There exists a numerical constant $c>0$ and a dimensional constant $c_n>0$ such that, for all rectangular parallelepipeds $R\subseteq \R^n$ and all measurable sets $E\subseteq R$ we have
\[
 \frac{w(E)}{w(R)} \leq c_n \Big(\frac{|E|}{|R|}\Big)^{(cn[w]_{A_\infty ^*})^{-1}}.
\]
\end{corollary}

\subsection{Embedding of \texorpdfstring{$A_\infty ^*$}{Ainfty*} into \texorpdfstring{$A_p ^*$}{Ap*}} The connection between Solyanik estimates and quantitative embeddings of $A_\infty ^*$ into the classes $A_p ^*$ was first presented in \cite{hp2}. Here we present the analogous result for multiparameter weights.

\begin{theorem} There exists some numerical constant $c>0$ such that, for all strong Muckenhoupt weights $w\in A_\infty ^*$ we have $w\in A_p ^*$ for all $p\geq e^{c[w]_{A_\infty^*}}$ and $[w]_{A_p ^*} \leq e^{e^{cp[w]_{A_\infty ^*}}}$.
\end{theorem}

\begin{proof}
We begin by fixing some weight $w\in A_\infty ^*$, $j\in\{1,\ldots,n\}$ and $\bar x^j \in \R^{n-1}$. Then for a.e. $\bar{x}^j\in\R^{n-1}$ the weight $w_{\bar x ^j}$ is an $A_\infty$ weight on the real line, uniformly in $\bar x^j$. By Lemma~\ref{l.weightsol1d} for $\gamma=0$, which is the one-dimensional version of Theorem~\ref{t.weightedsol}, we have the Solyanik estimate
\[
 \mathsf C_{1,w_{\bar x^j}} (\alpha)-1\leq 8(1-\alpha)^{(4[w_{\bar x^j}]_{A_\infty})^{-1}}\quad\text{whenever}\quad 1 > \alpha>1-e^{-4(\ln 8)[w_{\bar x^j}]_{A_\infty}}.
\]
Since $[w_{\bar x^j}]_{A_\infty} \leq [w]_{A_\infty ^*}$ for a.e. $ \bar x^j$ we get
\[
 \mathsf C_{1,w_{\bar x^j}} (\alpha)-1\leq 8(1-\alpha)^{(4[w]_{A_\infty ^*})^{-1}}\quad\text{whenever}\quad 1 > \alpha>1-e^{-4(\ln 8)[w]_{A_\infty ^*}},
\]
uniformly, for a.e. $\bar x^j$. Setting $\alpha_o\coloneqq 1-e^{-8(\ln 8) [w]_{A_\infty ^*}}$ we finally conclude that $ \mathsf C_{1,w_{\bar x^j}} (\alpha_o)\leq 1+8^{-2}$, uniformly for a.e. $\bar x^j$. Now a close examination of the proof of \cite{hlp}*{Theorem 6.1} shows that for every measurable set $E\subseteq \R$ we have
\[
\begin{split}
& w_{\bar x^j}(\{x \in \R :\, \MM \ind_E(x)>\lambda)\})
\\
&\qquad\qquad \leq \exp{\bigg[\log \mathsf C_{1,w_{\bar x^j}}(\alpha_o) \bigg( \bigg\lceil \frac{-\log\frac{\alpha_o}{\lambda}}{\log \alpha_o}  \bigg \rceil \bigg\lceil 2+ \frac{\log^+ (2\alpha_o)}{\log 1/\alpha_o} \bigg\rceil +1 \bigg)\bigg]}w_{\bar x^j}(E),
\end{split}
\]
where $\lceil x \rceil$ denotes the smallest positive integer which is no less than x. Thus for a.e. $\bar x^j\in \R^{n-1}$ we have that
\[
 w_{\bar x^j}(\{x\in\R:\, \MM  \ind_E(x)>\lambda\})\leq C_{1,w_{\bar x^j}}(\alpha_o)\frac{w(E)}{\lambda^{p_o}}\lesssim \frac{w(E)}{\lambda^{p_o}},
\]
where $p_o= \log (C_{1,w_{\bar x^j}}(\alpha_o)) e^{c[w]_{A_\infty ^*}}$ for some numerical constant $c>0$. However, this means that $\MM$ is of restricted weak type $(p_o,p_o)$ with respect to $w_{\bar x^j}$, uniformly for a.e. $\bar x^j$. By restricted weak type interpolation we conclude that $\MM$ maps $L^p(w_{\bar x^j})$ to itself with
\[
 \|\MM\|_{L^p(w_{\bar x^j})\to L^p(w_{\bar x^j})} \leq 2 \frac{p^\frac{1}{p}(C_{1,w_{\bar x^j}}(\alpha_o))^\frac{p_o}{p}}{(p-p_o)^\frac{1}{p}}
\]
for $p>p_o$. From this we conclude that $\|\MM\|_{L^{2p_o}(w_{\bar x^j})\to L^{2p_o}(w_{\bar x^j}) }\leq 4(C_{1,w_{\bar x^j}}(\alpha_o))^\frac{1}{2}.$
Now Riesz-Thorin interpolation, applied to a linearization of $\MM$ gives the bound
\[
 \|\MM\|_{L^{p}(w_{\bar x^j})\to L^{p}(w_{\bar x^j}) }\leq 4^\frac{2p_o}{p} (C_{1,w_{\bar x^j}}(\alpha_o))^\frac{p_o}{p}.
\]
for $p>2p_o$. We now remember the lower bound
\[
\|\MM\|_{L^p(\upsilon)\to L^{p}(\upsilon)}\gtrsim_p [\upsilon]_{A_p} ^\frac{1}{p},\quad p\in(1,\infty),
\]
valid for all one-dimensional weights $\upsilon\in A_p$. This is a simple consequence of the definition of the $A_p$-constant; the details are in \cite{Muck}. We conclude that $w_{\bar x^j }\in A_p$ for all $p>e^{c[w]_{A_\infty ^*}}$ and $[w_{\bar x^j}]_{A_p} \lesssim_p \exp(\exp (c[w]_{A_\infty ^*}))$ for some numerical constant $c>0$. Since these bounds are uniform in $j\in\{1,\ldots,n\}$ and $\bar x^j \in \R^{n-1}$ this concludes the proof of the theorem.
\end{proof}

\subsection{A weighted covering lemma for rectangles} We close the discussion on applications of weighted Solyanik estimates by providing a covering lemma for rectangles in $\R^n$ under the presence of $A_p ^*$-weights. This is an immediate application of our results. The formulation that follows might moreover turn out to be useful for future reference. Note that the statement of the corollary is given with respect to the ``rectangular'' $A_p ^*$-constants, $[w]_{A_p ^\mathsf{rec}}$.
\begin{corollary}\label{c.weightedcover} Let $\{R_j\}_{j=1} ^N$ be a finite collection of rectangular parallelepipeds in $\R^n$ whose sides are parallel to the coordinate axes, $w\in A_p ^*$ for some $p\in[1,\infty)$ be a strong Muckenhoupt weight in $\R^n$, and $\delta\in(0,e^{-cn[w]_{A_p ^\mathsf{rec}}})$ be a parameter. There exists a subcollection $\{\tilde R_k\}_{k=1} ^M\subseteq \{R_j\}_{j=1} ^N$, such that
\begin{itemize}
 \item[(i)] We have
\[
 w\big(\bigcup_j R_j \big) \leq \big(1+c_n\delta^{(cn[w]_{A_p ^\mathsf{rec}})^{-1}}\big) w\big(\bigcup_k \tilde R_k\big).
\]
 \item[(ii)] The rectangles in the collection $\{\tilde R_k\}_k$ are sparse in the sense that
\[
 \sum_k w(\tilde R_k) \leq\frac{[w]_{A_p ^\mathsf{rec} } }{\delta^p}w\big(\bigcup_k \tilde R_k\big)\;.
\]
\end{itemize}
Here $c>0$ is a numerical constant and $c_n>0$ depends only on the dimension.
\end{corollary}
\begin{proof} We perform the standard C\'ordoba-Fefferman selection algorithm from \cite{CF}. Thus we define $\tilde R_1\coloneqq R_1$ and let us assume that we have chosen $\tilde R_1,\ldots,\tilde R_j\eqqcolon R_J$. We then choose $\tilde R_{j+1}$ to be the first rectangle $R$ among the ones in the list $\{R_{J+1},\ldots,R_N\}$ that satisfies
\[
 \big|R\cap \bigcup_{\ell\leq j} \tilde R_j\big|\leq (1-\delta) |R|\;.
\]
If no such rectangle exists the selection algorithm terminates. Suppose now that $R\in\{R_j\}_j$ were not selected. Then there exists $k\leq M$ such that
\[
 \big|R\cap \bigcup_{\ell\leq k} \tilde R_\ell\big|>(1-\delta)|R|
\]
and thus
\[
 \bigcup_{j=1} ^N R_j \subseteq\big\{ x\in\R^n:\, \Ms(\ind_{\cup_k \tilde R_k})(x)>1-\delta\big\}.
\]
Now since for one dimensional weights we have $[w]_{A_\infty}\leq [w]_{A_p}$, see \cite{HytP}, it follows that $[w]_{A_\infty ^*}\leq [w]_{A_p ^\mathsf{rec}}$ for all $p\in[1,\infty)$. Therefore, for $\delta < e^{-cn[w]_{A_p ^\mathsf{rec}}} \leq e^{-cn[w]_{A_\infty ^\mathsf{rec}}}$ we have by Theorem~\ref{t.weightedsol}
\[
 w\big(\bigcup_j R_j \big) \leq \big(1+c_n \delta^{(cn[w]_{A_\infty ^*})^{-1}} \big)w\big(\bigcup_k \tilde R_k\big)\leq \big(1+c_n \delta^{(cn[w]_{A_p ^\mathsf{rec}})^{-1}} \big)w\big(\bigcup_k \tilde R_k\big),
\]
hence the proof of (i) is complete.

Now we define the increments $\tilde E_0\coloneqq \tilde R_0$ and $\tilde E_k \coloneqq \tilde R_k \setminus \cup _{\ell<k} \tilde R_\ell$ so that the $\tilde E_k$'s are disjoint and $\cup_k \tilde R_k=\cup_k \tilde E_k$. Note that the selection algorithm guarantees that $|  \tilde E_k|\geq \delta |	\tilde R_k|$. Since $w\in A_p ^*$ we also have
\[
 \delta^p\leq \Big(\frac{|\tilde E_k|}{|\tilde R_k|}\Big)^p \leq [w]_{A_p ^\mathsf{rec}} \frac{w(\tilde E_k)}{w(\tilde R_k)}.
\]
Thus
\[
 \sum_j w(\tilde R_k)\leq \frac{[w]_{A_p ^\mathsf{rec} } }{\delta^p} \sum_k w(\tilde E_k) = \frac{[w]_{A_p ^\mathsf{rec} } }{\delta^p} w\big(\bigcup_k \tilde R_k \big)
\]
as desired.
\end{proof}

\section{Solyanik estimates with respect to weights}
In this section we give the proof of Theorem~\ref{t.solwrtw}. The idea of the proof is very simple and bypasses all the problems that can be caused by the fact that, in the definition of $\Mw$, the presence of $w$ couples the variables making it technically hard to develop inductive arguments as the one in the proof of Theorem~\ref{t.weightedsol}. An inductive proof for the bound $\Mw:L^p(w)\to L^p(w)$ is however possible. See for example \cite{F} and \cite{LOSH}. Here we adopt a different approach and use the hypothesis $w\in A_\infty ^*$ in order to obtain Solyanik estimates for $\Mw$ by the weighted Solyanik estimates for $\Ms$.
\begin{proof}[Proof of Theorem~\ref{t.solwrtw}] Let $\alpha\in(0,1)$, $w\in A_\infty ^*$, and let $x\in E_\alpha\coloneqq \{x\in\R^n:\, \Mw(\ind_E)(x)>\alpha\}$. There exists a rectangular parallelepiped $R_x$ such that $w(R_x\cap E)/w(R_x)>\alpha$ and $x\in R_x$. Since $w\in A_\infty ^*$ there exists $1\leq p_o<+\infty$ such that $w\in A_{p_o} ^*$. Then $w$ has the property that for each rectangular parallelepiped $R\subseteq \R^n$ and each measurable $A\subseteq R$ we have
\[
 \Big(\frac{|A|}{|R|}\Big)^{p_o} \leq [w]_{A_{p_o} ^{\mathsf{rec}}} \frac{w(A)}{w(R)}.
\]
Thus for each measurable $S\subset R$ we have
\[
 \frac{|S|}{|R|}\geq 1 - [w]_{A_{p_o} ^\mathsf{rec}} ^\frac{1}{p_o} \Big(1-\frac{w(S)}{w(R)}\Big)^\frac{1}{p_o}.
\]
Applying the inequality for $S\coloneqq R_x\cap E\subseteq R_x$ we can conclude
\[
  \frac{|E \cap R_x | }{|R_x|} \geq 1 -[w]_{A_{p_o} ^\mathsf{rec}} ^\frac{1}{p_o}(1-\alpha)^\frac{1}{p_o}\quad \text{for}\quad 1 > \alpha >1 -\frac{1}{[w]_{A_{p_o} ^\mathsf{rec}}}.
\]
Thus
\[
 E_\alpha \subseteq \big\{x\in\R^n:\, \Ms(\ind_E)(x)>1 -[w]_{A_{p_o} ^\mathsf{rec}} ^\frac{1}{p_o}(1-\alpha)^\frac{1}{p_o}\big\}
\]
and using Theorem~\ref{t.weightedsol} we get
\[
\begin{split}
 \Cw(\alpha)-1 & \leq \Cs\big(1- [w]_{A_{p_o} ^\mathsf{rec}} ^\frac{1}{p_o}(1-\alpha)^\frac{1}{p_o}\big)-1 \lesssim_n \big([w]_{A_{p_o} ^\mathsf{rec}} ^\frac{1}{p_o}(1-\alpha)^\frac{1}{p_o}\big)^{(c n[w]_{A_\infty ^*})^{-1}}
\\
& \lesssim_{w,n} (1-\alpha)^\frac{1}{c_{w,n}}
\end{split}
\]
for some $c_{w,n}>1$, as long as $\alpha$ is sufficiently close to $1$, depending only on $w$ and $n$.
\end{proof}

\section*{Acknowledgment} We would like to thank Michael Lacey for bringing to our attention the connection of multiparameter Solyanik estimates with versions of Journ\'e's lemma and the characterization of product \textsf{BMO} in terms of commutators, thus motivating several of the questions addressed in the present paper.

\end{document}